\documentclass[a4paper,12pt]{article}
\usepackage[T1]{fontenc}
\usepackage{mathrsfs}
\usepackage{graphicx}
\usepackage{enumerate}
\usepackage{multicol}
\usepackage{color}
\usepackage{amsmath,amssymb}
\usepackage{amsthm}
\usepackage{textcomp}
\usepackage{mathabx}

\usepackage{hyperref}

\usepackage{times}
\usepackage[varg]{txfonts}

\usepackage[top=1in, bottom=1in, left=0.75in, right=0.75in]{geometry}

\theoremstyle{plain}
\newtheorem{thm}{Theorem}
\newtheorem{lem}{Lemma}
\newtheorem{cor}{Corollary}
\newtheorem{prop}{Proposition}

\theoremstyle{definition}
\newtheorem{dfn}{Definition}

\theoremstyle{remark}

\newtheorem*{ackn}{Acknowledgment}

\newcommand{\N}{\mathbb{N}}
\newcommand{\R}{\mathbb{R}}

\newcommand{\grad}{\nabla}

\title{Characterizations of  second-order differential operators}
\author{Włodzimierz Fechner, Eszter Gselmann and Aleksandra \'Swiątczak}

\begin{document}

\maketitle

\begin{abstract}
 {Let $N, k$ be positive integers with $k\geq 2$, and $\Omega \subset \mathbb{R}^{N}$ be a domain.} By the well-known properties of the Laplacian and the gradient, we have 
\[
  \grad^2(f\cdot g)=g  \grad^2 f+f \grad^2 g+2\langle \grad f, \grad g\rangle 
\]
for all $f, g\in \mathscr{C}^{k}(\Omega, \mathbb{R})$.  {Due to the results of H.~K\"{o}nig and V.~Milman, \emph{Operator relations characterizing derivatives}.
Birkh\"{a}user / Springer, Cham, 2018.,} the converse is also true,  {i.e. this operator equation characterizes the Laplacian and the gradient} under some assumptions. Thus the main aim of this paper is to provide an extension of this result and to study the corresponding equation 
\[
  T(f\cdot g)= fT(g)+T(f)g+2B(A(f), A(g)) 
  \qquad 
  \left(f, g\in P\right), 
\]
where $Q$ and $R$ are commutative rings, $P$ is a subring of $Q$ and $T\colon P\to Q$ and $A\colon P\to R$ are additive, while  $B\colon R\times R\to Q$ is a symmetric and bi-additive. Related identities with one function will also be considered.
\end{abstract}

\section{Introduction}

Let $Q$ and $R$ be commutative rings and $P$ be a subring of $Q$. In this paper we will study additive mappings $T\colon P\to Q$ and $A\colon P\to R$  for which we have 
\[
\tag{$\Delta$}\label{Eq1}
 T(f\cdot g)= fT(g)+T(f)g+2B(A(f), A(g))
\]
for all $f, g\in P$ with an appropriate symmetric and bi-additive mapping $B\colon R\times R\to Q$. 



As we will see  equation \eqref{Eq1} plays an important role in the characterization of the Laplacian. 
Indeed, if $ N$ is a positive integer and $\Omega\subset \mathbb{R}^{N}$ is a domain (by a domain we mean a nonempty, open and connected set) and we set $P= \mathscr{C}^{2}(\Omega, \mathbb{R})$, $Q= \mathscr{C}(\Omega, \mathbb{R})$, $R=\mathscr{C}(\Omega, \mathbb{R}^{N})$ (by $\mathscr{C}^{k}(\Omega, \mathbb{R})$ with $k$ being a nonnegative integer we mean the space of real-valued $k$-times continuously differentiable functions on $\Omega$). Further, if
\[
 T(f)(x)=   \grad^2 f(x) = \sum_{i=1}^{N}\dfrac{\partial^{2} f}{\partial x_{i}^{2}}(x)
\qquad 
\left(f\in \mathscr{C}^{2}(\Omega, \mathbb{R}), x\in \Omega\right)
\]
and 
\[
 A(f)(x)= f'(x)= \grad f(x)=
 \begin{pmatrix}
  \dfrac{\partial f}{\partial x_{1}}(x)\\
  \vdots \\
  \dfrac{\partial f}{\partial x_{N}}(x)
 \end{pmatrix}
\qquad 
\left(f\in \mathscr{C}^{2}(\Omega, \mathbb{R}), x\in \Omega\right)
 \]
 and 
 \[
 B(u, v)= \langle u, v\rangle = \sum_{i=1}^{N}u_{i}v_{i}
 \qquad 
 \left(u, v\in \mathbb{R}^{N}\right), 
\]
then we have
\begin{multline*}
 T(f\cdot g)(x)=  \grad^2(f\cdot g)(x)=g(x)  \grad^2 f(x)+f(x)  \grad^2 g(x)+2\langle \grad f(x), \grad g(x)\rangle 
 \\
 =
 f(x)T(g)(x)+T(f)(x)g(x)+2B(A(f)(x), A(g)(x)) 
\end{multline*}
for all $f, g\in \mathscr{C}^{2}(\Omega, \mathbb{R})$ and $x\in \Omega$.

One of our main objectives is to prove the converse of this statement. 
More precisely, we intend to answer the question whether equation \eqref{Eq1} is appropriate to characterize second-order differential operators.
Moreover, if we substitute $f$ in place of $g$ in equation \eqref{Eq1}, then we get 
 \[
 \tag{$\ast$}\label{Eq2}
  T(f^{2})= 2fT(f)+2B(A(f), A(f)). 
 \]
Further, if we combine equations \eqref{Eq1} and \eqref{Eq2}, then by induction 
\[
\tag{$\bullet$}\label{bullet}
  T(f^{n})= nf^{n-1}T(f)+nB(A(f), A(f^{n-1}))
\]
follows for each fixed positive integer $n\geq 2$. 
As we will see, the answer to the first question we raised is affirmative, that is, equation \eqref{Eq1} is suitable for characterizing second-order differential operators defined on $\mathscr{C}^{k}(\Omega)$. 
After this, the question naturally arises as to whether instead of \eqref{Eq1} it is sufficient to require the much weaker identities \eqref{Eq2} or \eqref{bullet}  for the operators in question.

Here $Q$ and $R$ are assumed to be commutative rings such that $\mathrm{char}(Q)>2$ and $P$ be a subring of $Q$, while the unknown mappings $T \colon P\to Q$ and $A\colon P\to R$ are additive. Further, $B\colon R\times R\to Q$ is a given symmetric and bi-additive mapping. 

 At first, we consider the above-mentioned identities in a purely algebraic setting (i.e., we are not limited to the case when the above rings are function spaces). 
After that, the obtained results will be applied for $Q=R= \mathscr{C}(\Omega, \mathbb{R})$ and $P= \mathscr{C}^{k}(\Omega, \mathbb{R})$, providing characterization results for first- and second-order differential operators. Here $k$ is a nonnegative integer. 

The antecedents of this topic can be found in the paper \cite{FecSwi22}, where the authors considered  {a special case of the above problems with $B=0$}. Accordingly, the authors obtained characterization theorems for first-order differential operators. As we will see, these results are consequences of that of this paper. 

 In subsections \ref{SS1} and \ref{SS2} we summarize the necessary preliminaries from the theory of functional equation and of operator relations, respectively. Then, the results can be found in Section \ref{S2}.

\subsection{Preliminaries from the theory of functional equations}\label{SS1}

In this subsection, we collect  {briefly} the most important notions and statements related to generalized polynomials, which we will use in the second section. Here we follow the monograph Sz\'{e}kelyhidi \cite{Sze85}.  {In this subsection $S$ is a commutative semigroup and $G$ is a commutative group, both written additively.}

 {
For $n\in \mathbb{N}$ a function $A\colon S^{n}\to G$ is termed \emph{$n$-additive} if in each variable it is a semigroup homomorphism of $S$ into $G$.
If $n=1$ or $n=2$ then the function $A$ is called \emph{additive} or \emph{bi-additive}, respectively.}

 {
The \emph{diagonalization} or \emph{trace} of an $n$-additive
function $A\colon S^{n}\to G$ is defined as
 \[
  A^{\ast}(x)=A\left(x, \ldots, x\right)
  \qquad
  \left(x\in S\right).
 \]}
 {
The above notion can also be extended for the case $n=0$ by letting 
$S^{0}=S$ and by calling $0$-additive any constant function from $S$ to $G$. }

 {
One of the most important theoretical results concerning
multi-additive functions is the so-called \emph{Polarization
formula}, which says that every $n$-additive symmetric
function is \emph{uniquely} determined by its diagonalization under
some conditions on the domain and on the range. The
action of the {\emph{difference operator}} $\Delta$ on a function
$f\colon S\to G$ is defined as
\[\Delta_y f(x)=f(x+y)-f(x)
\qquad
\left(x, y\in S\right). \]
Polarization formula says that if $A\colon S^{n}\to G$ is a symmetric, $n$-additive function, then for all $m\in\mathbb{N}$ and
 $x, y_{1}, \ldots, y_{m}\in S$ we have
 \[
  \Delta_{y_{1}, \ldots, y_{m}}A^{\ast}(x)=
  \left\{
  \begin{array}{rcl}
   0 & \text{ if} & m>n \\
   n!A(y_{1}, \ldots, y_{m}) & \text{ if}& m=n.
  \end{array}
  \right.
 \]
In particular, for all $x, y\in S$ one has
$
  \Delta^{n}_{y}A^{\ast}(x)=n!A^{\ast}(y).
$}

The next corollary will be very useful for us in what follows. However, it is a consequence of the above observations.
\begin{cor}
\label{mainfact}
  Let $n\in \mathbb{N}$ and suppose that the multiplication by $n!$ is surjective in the commutative semigroup $S$ or injective in the commutative group $G$. Then for any symmetric, $n$-additive function $A\colon S^{n}\to G$, $A^{\ast}\equiv 0$ implies that
  $A\equiv 0$, as well.
\end{cor}

 {
A function $p\colon S\to G$ is called a \emph{generalized polynomial} if it has a representation as the sum of diagonalizations of symmetric multi-additive functions. In other words, $p$ is a 
 generalized polynomial \emph{of degree at most $n$} if and only if, it has the representation 
 \[
  p= \sum_{k=0}^{n}A^{\ast}_{k}, 
 \]
where $n$ is a nonnegative integer and $A_{k}\colon S^{k}\to G$ is a symmetric, $k$-additive function for each 
$k=0, 1, \ldots, n$.  Moreover, functions $p_{k}\colon S\to G$ of the form 
$ p_{k}= A_{k}^{\ast}$, 
where $A_{k}\colon S^{n}\to G$ are symmetric and $k$-additive are called \emph{generalized monomials of degree $k$}. 
Obviously, generalized monomials of degree $0$ are constant functions and generalized monomials of degree $1$ are additive functions. }

 {
Now, assume that $X$ and $Y$ are linear-topological spaces, $Y$ is partially ordered,  $D$ is a nonvoid, open and convex subset of $X$, $C \subset X$ is a proper cone in $X$, 
{i.e., a subset of $X$ that is closed under positive scalar multiplication} and $n \in \N$. A map $f\colon D \to Y$ is termed \emph{$n$-convex} if $\Delta^{n+1}_h f(x) \geq 0,$ for all $x \in D$ and $h \in C$ such that {$x + i h \in D$} for all $i=1, 2, \dots , n+1$.
It is clear that every generalized polynomial of degree at most $n$ is an $n$-convex function.}

The following theorem is a special case of a result of R. Ger \cite[Theorem 6]{Ger74}.
\begin{thm}[R. Ger]
Assume that $X$ and $Y$ are linear-topological spaces and $Y$ is partially ordered.
If $f\colon D \to Y$ is an $n$-convex function defined on an open and convex subset of $X$ which is bounded on a set of the second category with the Baire property. Then $f$ is continuous on $D$.
\end{thm}

 {
The above theorem applies a fortiori to generalized polynomials. In particular, if the set of zeros of a generalized polynomial is of the second category with the Baire property, then it is closed. Consequently, it has a nonempty interior and one can apply a result of L. Sz\'{e}kelyhidi \cite{Sze85}, which says that if a generalized polynomial vanishes on a nonempty open set, then it vanishes everywhere. Thus, we have the following corollary.}
\begin{cor}\label{kat}
Assume that $X$ and $Y$ are linear-topological spaces and $Y$ is partially ordered.
If $f\colon X \to Y$ is a generalized polynomial which is equal to zero on a set of the second category with the Baire property, then $f\equiv 0$ on $X$.
\end{cor}

\subsection{Preliminaries from the theory of operator relations}\label{SS2}

In this subsection we use the notation and the terminology of the monograph K\"{o}nig--Milman \cite{KonMil18}, see also \cite{KonMil11, KonMil12a, KonMil12, KonMil13, KonMil14}. 

\begin{dfn}
 Let $k$ be a nonnegative integer, $N$ and $m$ be positive integers and $\Omega \subset \mathbb{R}^{N}$ be an open set. An operator $A\colon \mathscr{C}^{k}(\Omega, \mathbb{R})\to \mathscr{C}(\Omega, \mathbb{R}^{m})$ is \emph{non-degenerate} if for each nonvoid open subset $U\subset \Omega$ and all $x\in U$, there exist functions $g_{1}, \ldots, g_{m+1}\in \mathscr{C}^{k}(\Omega, \mathbb{R})$ with supports in $U$ such that the vectors $(g_{i}(x), A g_{i}(x))\in \mathbb{R}^{m+1}$, $i=1, \ldots, m+1$ are linearly independent in $\mathbb{R}^{m+1}$. 
\end{dfn}

\begin{dfn}
 Let $k$ and $N$ be positive integers with $k\geq 2$ and $\Omega\subset \mathbb{R}^{N}$ be an open set. We say that the operator $A\colon \mathscr{C}^{k}(\Omega, \mathbb{R})\to \mathscr{C}(\Omega, \mathbb{R})$ \emph{depends non-trivially on the derivative} if there exists $x\in \Omega$ and there are functions $f_{1}, f_{2}\in \mathscr{C}^{k}(\Omega, \mathbb{R})$ such that 
 \[
  f_{1}(x)= f_{2}(x) \quad \text{and} \quad A f_{1}(x)\neq A f_{2}(x)
 \]
holds. 
\end{dfn}

In the second section we will utilize Theorem 7.1 from \cite{KonMil18} which is the statement below. 

\begin{thm}[H. K\"{o}nig, V. Milman]\label{Konig-Milman}
 Let $N$ be a positive integer, $k$ be a nonnegative integer and $\Omega \subset\mathbb{R}^{N}$ be a domain. Assume that $T, A\colon \mathscr{C}^{k}(\Omega, \mathbb{R})\to \mathscr{C}(\Omega, \mathbb{R})$ satisfy 
 \begin{equation}\label{Tfg}
  T(f\cdot g)= T(f)\cdot g+f\cdot T(g)+2A(f)\cdot A(g) 
  \qquad 
  \left(f, g\in \mathscr{C}^{k}(\Omega, \mathbb{R})\right)
 \end{equation}
and that $A$ is non-degenerate and depends non-trivially on the derivative. 
Then there are continuous functions $a\in \mathscr{C}(\Omega, \mathbb{R})$ and $b, c\in \mathscr{C}(\Omega, \mathbb{R}^{N})$ such that we have 
\[
 \begin{array}{rcl}
  T(f)(x)&=& \langle f''(x)c(x), c(x)\rangle +R(f)(x)\\[2mm]
  A(f)(x)&=& \langle f'(x), c(x)\rangle
 \end{array}
\qquad 
\left(f\in \mathscr{C}^{k}(\Omega, \mathbb{R}), x\in \Omega\right), 
\]
where 
\[
 R(f)(x)= \langle f'(x), b(x)\rangle +a(x)f(x)\ln \left(|f(x)|\right) 
 \qquad 
 \left(f\in \mathscr{C}^{k}(\Omega, \mathbb{R})\right). 
\]
If $k=1$, then necessarily $c\equiv 0$. Further, if $k=0$, then necessarily $b\equiv 0$ and $c\equiv 0$.

Conversely, these operators satisfy the above second-order Leibniz rule. 
\end{thm}

\begin{cor}\label{Cor_KM}
Under the assumptions of the previous theorem, suppose that the operators $T, A\colon \allowbreak \mathscr{C}^{k}(\Omega, \allowbreak \mathbb{R}) \allowbreak \to \mathscr{C}(\Omega, \mathbb{R})$ are additive. Then $a\equiv 0$ in the above theorem. In other words, there are continuous functions $b, c\in \mathscr{C}(\Omega, \mathbb{R}^{N})$ such that we have 
\begin{equation}
 \begin{array}{rcl}
  \label{TA} T(f)(x)&=& \langle f''(x)c(x), c(x)\rangle + \langle f'(x), b(x)\rangle \\[2mm]
   A(f)(x)&=& \langle f'(x), c(x)\rangle
 \end{array}
\qquad 
\left(f\in \mathscr{C}^{k}(\Omega, \mathbb{R}), x\in \Omega\right). 
\end{equation}
If $k=1$, then necessarily $c\equiv 0$. Further, if $k=0$, then necessarily $b\equiv 0$ and $c\equiv 0$. 

Conversely, these operators are additive and also satisfy the above second-order Leibniz rule. 
\end{cor}

The assumption that $A$ is non-degenerate and depends non-trivially on the derivative, is essential in Theorem \ref{Konig-Milman} and also in Corollary \ref{Cor_KM}. Our example below is a slight modification of the example on page 114 of \cite{KonMil18}. 
To see the necessity of the above-mentioned conditions, suppose that $\Omega =\mathbb{R}^{N}$ and let $h\in \mathbb{R}^{N}$ be an arbitrary nonzero vector. Consider the difference operators 
$T, A\colon \mathscr{C}^{k}(\Omega, \mathbb{R})\to \mathscr{C}(\Omega, \mathbb{R})$  defined by 
$T f= \Delta_{h}f$ and
$A f= -\frac{\sqrt{2}}{2}\Delta_{h}f$
for all $f\in \mathscr{C}^{k}(\Omega, \mathbb{R})$. 
Then we have 
\begin{align*}
 T(f\cdot g)(x)
 & = f(x+h)g(x+h)-f(x)g(x)\\
& =  f(x)g(x+h)-f(x)g(x)+f(x+h)g(x)-f(x)g(x)\\
& +f(x)g(x)-g(x+h)f(x)-f(x+h)g(x)+g(x+h)f(x+h) \\
& =f(x)\cdot \left[g(x+h)-g(x)\right]+\left[f(x+h)-f(x)\right]\cdot g(x) \\
& + 2 \frac{\sqrt{2}}{2} \left[f(x)-f(x+h)\right] \cdot \frac{\sqrt{2}}{2} \left[g(x)-g(x+h)\right]  \\
& = f(x)\cdot T(g)(x)+T(f)(x) \cdot g(x)+2 A(f)(x)\cdot A(g)(x)
\end{align*}
for all $f, g\in \mathscr{C}^{k}(\Omega, \mathbb{R})$ and $x \in \Omega$. At the same time, the operators $T$ and $A$ are obviously not of the form as stated in Theorem \ref{Konig-Milman}. This is caused by the fact that in this example $A$ is not non-degenerate (and neither does $T$).  {This follows from a more general example of \cite[Chapter 8, pages 141-142]{KonMil18}. }

\section{Results}\label{S2}

In this section, we first examine equations \eqref{Eq1}, \eqref{Eq2} and \eqref{bullet} from a purely algebraic point of view. 
The main point of Lemma \ref{Lem2} below is that equations \eqref{Eq2} and \eqref{bullet} are separately equivalent to equation \eqref{Eq1} for additive operators. Next, we prove some extension theorems. Finally, by using these findings and the results of K\"{o}nig--Milman \cite{KonMil18}, we obtain characterization theorems for first- and second-order differential operators.

\begin{lem}\label{Lem2}
Let $n\geq 2$ be a positive integer, $Q$ and $R$ be commutative rings such that $\mathrm{char}(Q)>n!$, and $P$ be a subring of $Q$. Let further $T \colon P\to Q$ and $A\colon P\to R$ be additive mappings for which  {equation \eqref{bullet} holds }
for all $f\in P$ with a symmetric and bi-additive mapping $B\colon R\times R\to Q$. If  {$n=2$ or} $A(1)=0$, then equation \eqref{Eq1} is satisfied
for all $f, g\in P$. 
\end{lem}

\begin{proof}
 {Under the assumptions of the lemma, let us consider }
the mapping $\mathcal{A}_{n}\colon P^{n}\to Q$ defined through 
\begin{multline}\label{Eq_lemma1}
 \mathcal{A}_{n}(f_{1}, \ldots, f_{n})
 \\
 = T(f_{1}\cdots f_{n})- 
 \left[ 
 f_{2}\cdots f_{n}T(f_{1})+ f_{1}f_{3} \cdots f_{n}T(f_{2})+ \cdots + f_{1}\cdots f_{n-1}T(f_{n}) \right]
 \\
  -
 \left[\sum_{i=1}^{n}B(A(f_{i}), A(f_{1}\cdots f_{i-1}\cdot f_{i+1}\cdots f_{n}))\right]
 \\
 \left(f_{1}, \ldots, f_{n}\in P\right). 
\end{multline}
Due to the additivity of the operators $T$ and $A$ and the bi-additivity of $B$, the mapping $\mathcal{A}_{n}$ is symmetric and $n$-additive. Next, by equation \eqref{bullet} its trace is identically zero.
 {Since $\mathrm{char}(Q)>n!$, the multiplication by $n!$ is injective on $Q$. Thus Corollary \ref{mainfact} implies that $\mathcal{A}_{n}$ is identically zero}. 
 {If $n=2$, then we arrive at
$$
 0 = \mathcal{A}_{2}(f_{1}, f_{2}) = T(f_1\cdot f_2) - f_1T(f_2) -  f_2T(f_1)  {-2B(A(f_{1}), A(f_{2}))},
$$
for all $f_1, f_2 \in P$, which completes the proof in this case.
In case of arbitrary $n > 2$ take}
$f\in P$ and substitute  {$f_{1}=f$ and }$f_{i}=1$ for $i=2, \ldots, n$  {in \eqref{Eq_lemma1}}, 
\[
 T(f)= T(f)+(n-1)fT(1)+nB(A(f), A(1)), 
\]
i.e. 
\[
 (n-1)T(1)f+nB(A(f), A(1))=0.
\] 
Since $A(1)=0$ and due to the bi-additivity of $B$, we obtain 
\[
 B(A(f), A(1))= B(A(f), 0)=0.
\]
From this $T(1)=0$ follows. 
Again, since the mapping $\mathcal{A}_{n}$ is identically zero, if $f, g\in P$ and we substitute  {$f_{1}=f, f_{2}=g$ and} $f_{i}=1$ for $i=3, \ldots, n$  {in \eqref{Eq_lemma1}}, we get that 
\begin{multline*}
 T(fg)= fT(g)+gT(f)+(n-2)fgT(1)+\left[2B(A(f), A(g))+(n-2)B(A(fg), A(1))\right]
 \\
 =
 fT(g)+gT(f) +2B(A(f), A(g))
\end{multline*}
is fulfilled for all $f, g\in P$. 
\end{proof}

As a consequence of this lemma   {and its proof}, now we show that if the rings $Q$ and $R$ are topological rings, then if equation \eqref{bullet} is fulfilled only on a subset $U\subset P$ with a nonempty interior (however mappings $T$ and $A$ are defined on the whole ring $P$), then this also implies that equation \eqref{Eq1} is satisfied everywhere on $P$.

\begin{cor}\label{Cor2}
Let $n\geq 2$ be a positive integer, $Q$ and $R$ be commutative topological rings such that $\mathrm{char}(Q)>n!$, let further $P$ be a subring of $Q$ and $U\subset P$ be a nonvoid open set. Let further $T \colon P\to Q$ and $A\colon P\to R$ be additive mappings for which equation \eqref{bullet}
holds for all $f\in U$ with a symmetric and bi-additive mapping $B\colon R\times R\to Q$. If   {$n=2$ or} $A(1)=0$, then equation \eqref{Eq1} is satisfied
for all $f, g\in P$. 
\end{cor}
\begin{proof}
 {Under the assumptions of the corollary, 
consider the mapping $\mathcal{A}_n\colon P\times P\to Q$ defined as in the proof of Lemma \ref{Lem2}.
Repeating the calculations from the proof of Lemma \ref{Lem2} we get that the diagonalization of $\mathcal{A}_n$ vanishes on the open subset $U$ of $P$. In view of Sz\'{e}kelyhidi \cite{Sze85} this is possible only if this map is identically zero on $P$, yielding in turn that $\mathcal{A}_n$ is identically zero on $P\times P$. So, repeating the respective parts of the proof of Lemma \ref{Lem2}, we derive that \eqref{Eq1} is satisfied for all $f, g\in P$. }
\end{proof}

We can join this corollary with Corollary \ref{kat} to relax the assumptions upon the set of zeros if additionally $P$ and $R$ are linear-topological spaces. Therefore, we assume that $P$ and $R$ are topological algebras over $\R$.  {Note that any real algebra trivially satisfies the assumptions of injectivity and surjectivity of scalar multiplication.}

\begin{cor}\label{Cor2k}
Let $Q$ and $R$ be topological algebras over $\R$, $Q$ being partially ordered. Further, 
$P$ be a subring of $Q$ and $U\subset P$ be a a set of the second category in $Q$ with the Baire property.
Let further $T \colon P\to Q$ and $A\colon P\to R$ be additive mappings for which  {equation \eqref{bullet}}
holds for all $f\in U$ with a symmetric and bi-additive mapping $B\colon R\times R\to Q$. If  {$n=2$ or} $A(1)=0$, then equation \eqref{Eq1} is satisfied
for all $f, g\in P$. 
\end{cor}

Note that the technique described above is suitable also for characterizing first-order differential operators. Indeed, from Lemma \ref{Lem2} we immediately deduce with the choice $B\equiv 0$ the following statement:

  \emph{Let $Q$ be a commutative ring such that $\mathrm{char}(Q)>2$ and $P$ be a subring of $Q$. Let further $T \colon P\to Q$ be an additive mapping for which we have 
\begin{equation}\label{Eq_spec2}
T(f^{n})=nf^{n-1}\cdot T(f)
\end{equation}
for all $f\in P$ with a symmetric and bi-additive mapping $B\colon R\times R\to Q$. 
Then equation the mapping $T$ also fulfils 
\[
 T(f\cdot g)= f\cdot T(g)+T(f)\cdot g
\]
for all $f, g\in P$.}

 Further, if we combine this statement with \cite[Theorem 3.5]{KonMil18} the obtain the following characterization of first-order differential operators. In fact, the following proposition was also proved in \cite{FecSwi22} with a different method. 

\begin{prop}
 Let $N\in \mathbb{N}$, $k$ be a nonnegative integer and $\Omega \subset \mathbb{R}^{N}$ be a nonempty open set. If the additive operator $T\colon \mathscr{C}^{k}(\Omega, \mathbb{R})\to \mathscr{C}(\Omega,  \mathbb{R})$ satisfies  {\eqref{Eq_spec2}}
 for all $f\in \mathscr{C}^{k}(\Omega)$, then there exists a continuous function $b\in \mathscr{C}(\Omega, \mathbb{R}^{N})$ such that 
\[\label{Tb}
 T(f)(x)= \langle b(x), f'(x) \rangle 
 \qquad 
 \left(f\in \mathscr{C}^{k}(\Omega, \mathbb{R})\right). 
\]
For $k=0$, the function $b$ is necessarily identically zero. 
\end{prop}

In what follows, we will apply the above results by choosing $P, Q$ and $R$ to be a given function space, thus obtaining characterization theorems for second-order differential operators.

\begin{thm}
  Let $N$ be a positive integer, $n\geq 2$ be a positive integer, $k$ be a nonnegative integer and $\Omega \subset\mathbb{R}^{N}$ be a domain. Assume that the additive operators $T, A\colon \mathscr{C}^{k}(\Omega, \mathbb{R})\to \mathscr{C}(\Omega, \mathbb{R})$ satisfy 
 \begin{equation}\label{Eq6}
  T(f^{n})= nf^{n-1}T(f)+n A(f)A(f^{n-1})
  \qquad 
  \left(f\in \mathscr{C}^{k}(\Omega, \mathbb{R})\right)
 \end{equation}
and that $A$ is non-degenerate, depends non-trivially on the derivative and annihilates the identically $1$ function. 
If  {$n=2$ or $A(1)=0$}, then there are continuous functions $b, c\in \mathscr{C}(\Omega, \mathbb{R}^{N})$ such that formulas \eqref{TA} are satisfied.
Further, if $k=1$, then necessarily $c\equiv 0$. Further, if $k=0$, then necessarily $b\equiv 0$ and $c\equiv 0$. 

Conversely, these operators are additive and also satisfy equation \eqref{Eq6}.
\end{thm}
\begin{proof}
 Consider the rings $P= \mathscr{C}^{k}(\Omega, \mathbb{R})$, $Q= \mathscr{C}(\Omega, \mathbb{R})$ and $R= \mathscr{C}(\Omega, \mathbb{R})$ with the pointwise addition and multiplication of functions and also the symmetric and bi-additive mapping $B\colon \mathscr{C}(\Omega, \mathbb{R})\times \mathscr{C}(\Omega, \mathbb{R})\to \mathscr{C}(\Omega, \mathbb{R})$ defined through 
 \[
  B(f, g)= f\cdot g 
  \qquad 
  \left(f, g\in \mathscr{C}(\Omega, \mathbb{R})\right). 
 \]
Then the assumptions of Lemma \ref{Lem2} are fulfilled. 
This yields that the equality \eqref{Tfg} holds for all $f, g\in \mathscr{C}^{k}(\Omega, \mathbb{R})$. 
In view of Theorem \ref{Konig-Milman}, this implies that there are continuous functions $b, c\in \mathscr{C}(\Omega, \mathbb{R}^{N})$ such that formulas \eqref{TA} are fulfilled. The converse implication is a routine calculation.
\end{proof}

 {In our last result we will deal with a conditional differential relation that is meaningful for  {nowhere} vanishing mappings. In the proof first we will show that formulas \eqref{TA} are valid for mappings $f$ such that $f^2 - 1$ is  {nowhere} zero. Then we will apply Corollary \ref{Cor2} to extend \eqref{TA} to each $f\in \mathscr{C}^{k}(\Omega, \mathbb{R})$. }

\begin{thm}
 Let $N$ be a positive integer, $k$ be a nonnegative integer and $\Omega \subset \mathbb{R}^{N}$ be a domain. Let $T, A\colon \mathscr{C}^{k}(\Omega, \mathbb{R})\to \mathscr{C}(\Omega, \mathbb{R})$ be  additive mappings such that $A(1)=0$ and 
 \begin{equation}\label{Eq7}
   T\left(\frac{1}{f}\right)= -\frac{1}{f^{2}}T(f)+\frac{2}{f^{3}}A(f)^{2}
 \end{equation}
 holds  for all $f\in \mathscr{C}^{k}(\Omega, \mathbb{R})$ for which $f(x)\neq 0$ for all $x\in \Omega$. 
Then there are continuous functions $b, c\in \mathscr{C}(\Omega, \mathbb{R}^{N})$ such that formulas \eqref{TA} hold true.
Further, if $k=1$, then necessarily $c\equiv 0$. Further, if $k=0$, then necessarily $b\equiv 0$ and $c\equiv 0$. 
\end{thm}

\begin{proof}
 {Under the assumptions of the theorem, if} we substitute the identically $1$ function in place of $f$ in equation \eqref{Eq7}, then 
\[
 T(1)=-T(1)+2A(1)^{2}
\]
follows. Since $A(1)=0$, then $T(1)=0$, as well.

Let now $f\in \mathscr{C}^{k}(\Omega, \mathbb{R})$ be a function for which we have 
\[
  {f(x)^{2}} \neq 1
\]
for all $x\in \Omega$. 
Since $\Omega$ is a connected set and $f$ is continuous, then we have precisely three options: either $f>1$, $f(x) \in (-1,1)$ for all $x \in \Omega$, or $f<-1$.
If we put $f^{2}-1$ in place of $f$ in equality \eqref{Eq7}, then 
\[
  T\left(\frac{1}{f^{2}-1}\right)= -\frac{1}{(f^{2}-1)^{2}}T(f^{2}-1)+\frac{2}{(f^{2}-1)^{3}}A(f^{2}-1)^{2},
\]
 or after multiplying both sides of this equation with $2$:
\[
  T\left(\frac{2}{f^{2}-1}\right)= -\frac{2}{(f^{2}-1)^{2}}T(f^{2}-1)+\frac{4}{(f^{2}-1)^{3}}A(f^{2}-1)^{2}.
\]
At the same time, due to the additivity of the operator $T$, we have 
\[
  T\left(\frac{2}{f^{2}-1}\right)= T\left(\frac{1}{f-1}-\frac{1}{f+1}\right)=
  T\left(\frac{1}{f-1}\right)-T\left(\frac{1}{f+1}\right)
\]
Moreover, equation \eqref{Eq7} appplied for $f-1$ and $f+1$, resp. yield that 
\[
 T\left(\frac{1}{f-1}\right)= -\frac{1}{(f-1)^{2}}T(f-1)+\frac{2}{(f-1)^{3}}A(f-1)^{2}
\]
and 
\[
 T\left(\frac{1}{f+1}\right)= -\frac{1}{(f+1)^{2}}T(f+1)+\frac{2}{(f+1)^{3}}A(f+1)^{2}. 
\]
Combining the above identities, we get
\begin{multline*}
 -\frac{2}{(f^{2}-1)^{2}}T(f^{2}-1)+\frac{4}{(f^{2}-1)^{3}}A(f^{2}-1)^{2}
 \\
 =
  T\left(\frac{2}{f^{2}-1}\right)
 =
   T\left(\frac{1}{f-1}\right)-T\left(\frac{1}{f+1}\right)
  \\
  =
   \left[-\frac{1}{(f-1)^{2}}T(f-1)+\frac{2}{(f-1)^{3}}A(f-1)^{2}\right]
  \\
  -
   \left[-\frac{1}{(f+1)^{2}}T(f+1)+\frac{2}{(f+1)^{3}}A(f+1)^{2} \right].
\end{multline*} 
Since $A(1)=0$ and $T(1)=0$, after some rearrangement we deduce that 
\begin{multline*}
 -\frac{2}{(f^{2}-1)^{2}}T(f^{2})+\frac{4}{(f^{2}-1)^{3}}A(f^{2})^{2}
  \\
  =
 -\frac{1}{(f-1)^{2}}T(f)+\frac{2}{(f-1)^{3}}A(f)^{2}
 +\frac{1}{(f+1)^{2}}T(f)-\frac{2}{(f+1)^{3}}A(f)^{2} 
\end{multline*}
i.e., 
\[
 -\frac{2}{(f^{2}-1)^{2}}T(f^{2})+\frac{4}{(f^{2}-1)^{3}}A(f^{2})^{2}
 =
 -\frac{4f}{(f^2-1)^{2}}T(f)+\frac{12f^2+4}{(f^{2}-1)^3} A(f)^{2}. 
\]
Let us multiply now both sides of this equation with $(f^{2}-1)^{3}$ to get that 
\[
 (f^2-1)\left[T(f^{2})-2fT(f)-2A(f)^{2}\right]= 2 \left[A(f^{2})^{2}-4f^{2}A(f)^{2}\right].
\]
It is clear that one can find at least five distinct rationals $\{r_1, \dots , r_5\}$ such that $r^2_if(x)^{2}\neq 1$ for all $x \in \Omega$ and each $i=1, \dots , 5$ and
let $$H = \{r\in\mathbb{Q} : r^2f(x)^{2}\neq 1\textrm{ for all } x \in \Omega\}.$$ Next, apply the last formula  for $f$ replaced by $rf$ with $r \in H$. Then due to the $\mathbb{Q}$-homogeneity of the operators $T$ and $A$, for all $r\in H$ we have 
\begin{multline*}
 r^{4}  {f^{2}} \left\{ \left[T(f^{2})-2fT(f)-2A(f)^{2}\right]-2\left[A(f^{2})^{2}-4f^{2}A(f)^{2}\right] \right\}
 \\
 + r^{2}\left\{ T(f^{2})-2fT(f)-2A(f)^{2} \right\} =0. 
\end{multline*}
Observe that the left-hand side of this equation is a (classical) polynomial with respect to $r\in H$ and it is identically zero. Since $H$ contains at least five elements,  all coefficients of this polynomial vanish. Especially, the coefficient corresponding to $r^{2}$ has to vanish, yielding that 
 {\eqref{Eq2} holds for any function $f\in \mathscr{C}^{k}(\Omega, \mathbb{R})$ for which $f^{2}(x)\neq 1$ for all $x\in \Omega$. Since the set 
\[
 U= \left\{f\in \mathscr{C}^{k}(\Omega, \mathbb{R}) \, \vert \, f^{2}(x)\neq 1 \text{ for all }x\in \Omega \right\}
\]
is a nonempty and open subset of $\mathscr{C}^{k}(\Omega, \mathbb{R})$,}
with the help of  {Corollaries \ref{Cor_KM} and \ref{Cor2}}, the statement of our theorem follows. 
\end{proof}

\begin{ackn}
 The research of E.~Gselmann has been supported by project no.~K134191 that has been implemented with the support provided by the National Research, Development and Innovation Fund of Hungary,
financed under the K\_20 funding scheme. 

The work of Aleksandra Świątczak is implemented under the project ``Curriculum for advanced doctoral education \&
training - CADET Academy of TUL'' co-financed by the STER Programme - Internationalization of
doctoral schools.

This article has been completed while one of the authors – Aleksandra Świątczak, was the
Doctoral Candidate in the Interdisciplinary Doctoral School at the Lodz University of Technology, Poland.
\end{ackn}

\end{document}